\documentclass[12pt, reqno]{amsart}
\usepackage{amsmath, amsthm, amscd, amssymb, graphicx, color, ulem,xcolor}
\usepackage{lmodern}
\usepackage{mathrsfs}
\usepackage{array}
\usepackage{enumerate}[(i).]
\usepackage{tabularx}
\usepackage{hyperref}
\newtheorem{theorem}{Theorem}[section]
\newtheorem*{theorem*}{Theorem}
\newtheorem{lemma}[theorem]{Lemma}
\newtheorem{proposition}[theorem]{Proposition}
\newtheorem{corollary}[theorem]{Corollary}
\theoremstyle{definition}
\newtheorem{definition}[theorem]{Definition}

\theoremstyle{remark}
\newtheorem{remark}[theorem]{Remark}
\numberwithin{equation}{section}
\newcommand{\li}{L^\infty}
\newcommand{\ck}{\mathcal{K}}

\begin{document}

	\title[The Skew Commutators of Toeplitz and Hankel operators]{The Skew Commutators of Toeplitz and Hankel operators in vector-valued Hardy Space}
	
	\author[Priyanka Aroda and Santanu Dey]{Priyanka Aroda$^*$ and Santanu Dey$^{**}$}
	\maketitle
	
	\begin{abstract}
            In this article, we characterize when a Toeplitz operator and a Hankel operator on the vector-valued Hardy space are skew commutators of each other, and determine necessary and sufficient conditions under which their product is self-adjoint. These characterizations extend the results in \cite{LZD}. In addition, we completely classify the skew commutators of the unilateral shift $S$, its adjoint $S^*$ and $S\oplus S^*.$ We also characterize the class of bounded linear operators that have $S$, $S^*$ and $S\oplus S^*$ as skew commutators.
    \end{abstract}
    
	\vspace{0.5cm}
    
	\vspace{0.5cm}

    \vspace{0.1in}


	\section{{Introduction and Preliminaries}}
    \noindent

    Toeplitz and Hankel operators have been among the most important classes of operators on several function spaces, and have been extensively studied. The rich algebraic and spectral structure has led to numerous applications in operator theory, complex analysis, mathematical physics, engineering, etc. We begin by introducing the notations that we are going to use during the discussion of the whole article.

     Let $\mathbb{D}$ denote the open unit disk $\{ z \in \mathbb{C}: |z| < 1 \}$ in the complex plane and let $H(\mathbb{D})$ denote the set of Holomorphic functions on $\mathbb{D}$. 
     The Hardy-Hilbert space, denoted by $H^2(\mathbb{D})$, 
     is defined as
    $$ H^2(\mathbb{D}):=\Bigl\{f\in H(\mathbb{D}): f(z)=\sum_{n=0}^{\infty}a_nz^n \text{ with } \sum_{n=0}^{\infty} |a_n|^2 <  \infty \Bigr\}, $$
    with inner product equipped with $\langle f,g\rangle=\sum a_n \overline b_n$, where $f(z)=\sum_{n=0}^{\infty}a_nz^n$  and  $g(z)=\sum_{n=0}^{\infty}b_nz^n$. 
    Let $H^\infty(\mathbb{D})$ denote the set of bounded analytic functions on $\mathbb{D}$. Furthermore, let $L^2$ and $L^\infty$ be the set of square integrable and essentially bounded functions on $\mathbb{T}:=\{ z \in \mathbb{C}: |z|=1 \}$ with respect to the normalized Lebesgue measure, respectively.

    By the virtue of Fatau's Theorem, one can identify $H^2(\mathbb{D})$ through the boundary limit as $H^2(\mathbb{T})$, which is a closed subspace of $L^2$ consisting of functions whose negative Fourier coefficients vanish. We use $H^2$ to denote both $H^2(\mathbb{D})$ and $H^2(\mathbb{T})$ according to the context, and $H^\infty = L^\infty \cap H^2$.
    
    In mid 1960s Brown-Halmos introduced the Toeplitz operator $T_\phi$ in the seminal paper \cite{BH}, and defined this as $$T_\phi:H^2\rightarrow H^2 \text{ by } T_{\phi}=PM_\phi|_{H^2}
    ,$$ where $\phi \in L^\infty$ and $P$ is the orthogonal projection of $L^2$ onto $H^2,$ also known as the Szeg\"o projection. The Hankel operator $H_\phi$ is defined as $$ H_\phi: H^2 \rightarrow H^2 \text{ by } H_\phi = P\mathcal{J}M_\phi|_{H^2}, $$ where $\mathcal{J}: L^2\rightarrow L^2$ is the unitary involution, known as the flip operator, which is defined by $\mathcal{J}f(z)=\bar{z}f(\bar{z})$.

    The most elementary Toeplitz operator is $T_z$, known as the forward shift operator, also denoted by $S$. This operator has been playing a significant role in function theory over the years. Brown-Halmos showed that a bounded operator $A$ acting on $H^2$ is a Toeplitz operator if and only if 
    $$ S^*AS=A,$$
    and in which case $ A=T_\phi, \text{ for some } \phi\in \li.$ Similarly, a bounded operator $A$, acting on $H^2$, is a Hankel operator if and only if it satisfies the relation 
    $$S^*A=AS.$$
    In that case $A=H_\phi,$ for some $\phi\in L^\infty$ due to Nehari's theorem \cite{ZN}.

    In an analogy to this classical setting, the Toeplitz and Hankel operators are defined on an arbitrary Hilbert space $\mathcal{H}$ by introducing a unilateral shift operator.

    \begin{definition}
        An isometry $S\in \mathcal{B}(\mathcal{H})$ is said to be a unilateral shift if there exists a subspace $\mathcal{L}$ of $\mathcal{H}$ such that $\mathcal{L} \perp S^n\mathcal{L}$ for $n\geq 1$, and $$\mathcal{H}=\bigoplus_{0}^{\infty}S^n\mathcal{L}.$$
       An operator $T\in \mathcal{B}(\mathcal{H})$ is said to be a Toeplitz operator if $S^*TS=T,$ and a Hankel operator if $S^*T=TS.$
    \end{definition}
    We observe that, unlike the classical case, these abstract definitions are formulated independently of any symbol. The subspace $\mathcal{L}$ in the above definition is called a \emph{wandering} subspace. For a unilateral shift $S\in \mathcal{B}(\mathcal{H})$, the wandering subspace is uniquely given by $\mathcal{H}\ominus S\mathcal{H}.$ 
    Let $D_T$ be the defect operator, defined as $D_{T} := \sqrt{I-T^*T},$ for $T\in \mathcal{B}(H)$. Then one can even identify the wandering subspace $\mathcal{L}$ as $D^2_{S^*}\mathcal{H},$ because for $h\in \mathcal{H}$, with $h=\sum_{i=0}^{\infty}S^il_i, ~l_i \in \mathcal{L}$, we have 
    $$ D^2_{S^*}(\sum_{i=0}^{\infty}S^il_i)=  (I-SS^*)(\sum_{i=0}^{\infty}S^il_i)=l_0. $$
    Hence, $D^2_{S^*}$ is nothing but the orthogonal projection of $\mathcal{H}$ onto the wandering subspace $\mathcal{L}.$

    Let $H(\mathbb{D}, \mathbb{C}^n)$ denote the set of Holomorphic function taking values in $\mathbb{C}^n.$ We now define a vector ($\mathbb{C}^n$) valued Hardy space 
    \begin{equation*}
        H^2_n:=\{f\in H(\mathbb{D}, \mathbb{C}^n): f(z)=\sum_{n=0}^{\infty}z^n a_n\text{ with } \sum_{n=0}^{\infty} ||a_n||^2 <  \infty \},
    \end{equation*}
    where the inner product is given by $\langle f,g\rangle=\sum a_n \overline b_n$, where $f(z)=\sum_{n=0}^{\infty}z^na_n$  and  $g(z)=\sum_{n=0}^{\infty}z^nb_n.$ $L^\infty(M_n(\mathbb{C}))$ be the algebra of essentially bounded functions on the unit circle taking values in $M_n(\mathbb{C}).$ Let $M_n(L^\infty)$ denote the collection of square matrices of order $n$ whose entries are coming from $L^\infty.$ Similarly, we define $H^\infty(M_n(\mathbb{C}))$ as the set of essentially bounded analytic functions on the unit circle taking values in $M_n(\mathbb{C})$ and $M_n(H^\infty)$ as collection of square matrices of order $n$ with entries in $H^\infty.$ We can identify
    $H^2_n, L^\infty(M_n(\mathbb{C}))$ and 
    $H^\infty(M_n(\mathbb{C}))$ as $\bigoplus_{1}^{n}H^2, M_n(L^\infty)$ and $M_n(H^\infty),$ respectively. The Toeplitz and Hankel operators are defined on $ H^2_n$ as follows:
    \begin{definition}
         For $F \in M_n(L^\infty)$, the Toeplitz operator $T_F : H^2_n\to H^2_n$ is defined as $$ T_F(h)(z)=P\big(F(z)h(z)\big), $$ 
         and the Hankel operator $H_F:H^2_n\to H^2_n$ is defined as $$  H_F(h)(z)=P\big(\overline{z}F(\overline{z})h(\overline{z})\big).$$
    \end{definition}
    
    The operator $H_F$ does not depend on the analytic part of $F$. i.e. $H_F=H_G$ if and only if $F-G \in M_n(H^\infty).$
    So whenever we consider the Hankel operator $H_F$, we will assume $F$ satisfies $(I-P)F = F$. For any $F \in M_n(L^\infty)$, 
    let $F^*$ denote the adjoint of $F$ and 
    $\widetilde F$ is defined as $F^*(\overline{z})$.
    It is known that, $T_{F_1}^*=T_{F_1^*}$ and $H_{F_2}^*=H_{\widetilde F_2}.$ As similar to the classical case, for $G \in M_n(L^\infty)$, we have
    $T_{F_1}T_G = T_{F_1 G}$.

   In 2000, Mart\'inez-Avenda\~no \cite{RAMA} characterized when a Toeplitz operator and a Hankel operator commute in the classical Hardy space $H^2$. In other words, they established the necessary and sufficient condition for the commutator $[T_\phi, H_\psi]$ to vanish, that is,
   $$[T_\phi, H_\psi]= T_\phi H_\psi - H_\psi T_\phi = 0 .$$ 
   This work subsequently inspired several related investigations, including those of Guo-Zheng \cite{GZ}, Gu \cite{CG}, Ding \cite{D}, Yan \cite{Y}, Lu-Kong \cite{LK}, Lee \cite{L}, and others.

   In particular, in 2003, Gu \cite{CG} extended Mart\'inez-Avenda\~no's characterization to the vector-valued Hardy space by establishing many fruitful results in the more general setting of abstract Hilbert spaces. 
   Recently, Li-Zheng-Ding \cite{LZD} introduced the notion of skew-commutator for bounded linear operators $A, B \in B(\mathcal{H})$ defined as
   $$_*[A, B]= AB- BA^*.$$
    \begin{definition}[\cite{LZD}]
        A bounded linear operator $B$ is said to the skew-commutator of $A$ if $ _*[A, B]=0.$
    \end{definition}
    They studied when a Toeplitz operator or a Hankel operator on the classical Hardy space $H^2$ are skew commutators of each other, and when their product is self-adjoint.

   The main objective of this article is to study the skew-commutator of the Toeplitz and Hankel operators in the vector-valued setting. It is worth noting that the operator-valued symbols give rise to noncommutativity, and consequently, many techniques employed in the scalar-valued setting cannot be carried over directly.
   
   In Section \ref{section skew commutators}, we establish the necessary and sufficient conditions for a Hankel operator to be the skew-commutator of a Toeplitz operator in the vector-valued Hardy space and vice versa. 
   For example, we have the following theorem.
   \begin{theorem*}[Theorem \ref{Th2}]
        Assume that $P(\overline z \phi(z))$ is regular. Then $T_\phi H_\psi =H_\psi T_\phi^*$  if and only if there exist $A\in M_n(\mathbb{C})$ such that the following holds:
        \begin{enumerate}[(i)]
            \item $\psi(z)-\phi ( \overline z)A$ is analytic,
            \item $\psi(z)+A\phi^*(z)$ is analytic, and
            \item $\phi(\overline z)A\phi(z)^*$ is analytic.
       \end{enumerate}
    \end{theorem*}
   Then, in Theorem \ref{Th3}, we obtain the self-adjoint characterization of the product of Toeplitz and Hankel operators. All these characterizations generalize the results of the same in the scalar-valued Hardy space achived by Li-Zheng-Ding in \cite{LZD}. 
   Alongside, in Theorem \ref{Th4},
   we also present a Toeplitz (Hankel) operator, being the skew-commutator of another Toeplitz (Hankel) operator, resulting in the vectorial generalization of the results obtained in \cite{LD}.

    The study of operators commuting with the unilateral shift on the Hardy space has been of interest. 
    It is well known that the class of commutants of unilateral shift is nothing but the algebra of analytic Toeplitz operators. This characterization underlines much of the subsequent theory of invariant subspaces and model theory. The skew commutators of the unilateral shift $S$ are nothing but the Hankel operators. Motivated by this, the following questions are natural to ask: 
    \begin{enumerate}
        \item What are the skew commutators of $S^*$?
        \item What are the operators whose skew commutator is $S$ (or $S^*$)?
    \end{enumerate}
    We give answers to the above questions in section \ref{section solutions of op eq}. 
    We show that the operators satisfying (2) are
    the block operators involving Toeplitz and Hankel operators.
    Additionally, we also find out the operator $X$ for which $SX$ is self adjoint. 
    Precisely, we have the following theorem.
    \begin{theorem*}[Theorem \ref{th1}]
        For a unilateral shift $S$ on $H^2$, we have
        \begin{enumerate}
            \item  $SX=XS^*$ if and only if $X=0$.
            \item $SX=X^*S^*$ if and only if $X=AS^*,$ for some self adjoint operator $A \in \mathcal{B}(H^2)$.
            \end{enumerate}
    \end{theorem*}
    Similarly, in Theorem \ref{p-3} and \ref{th2}, 
    we find out the operator 
    $X$ for which $XS, XS^*,S^*X$ are self adjoint.
    
     Timotin in \cite{DT} described the invariant subspaces of \(S \oplus S^{*}\) 
    in the scalar-valued Hardy space using methods from the 
    Sz.-Nagy--Foiaş theory. Gu and Luo \cite{CGSL} later considered the 
vector-valued operator
$S_{E} \oplus S_{F}^*$
acting on \(H_{E}^{2} \oplus H_{F}^{2}\). They expressed its invariant 
subspaces as kernels or ranges of block operators whose entries involve 
Toeplitz and Hankel operators. These block operators arise naturally as
intertwiners between \(S \oplus S\) and \(S \oplus S^{*}\). These results suggest studying the corresponding operator equations 
directly. In Theorem \ref{skew of S+S*}, we determine the skew commutator of 
$S \oplus S^*$ and describe its elements in block operator form. 
We also examine the related inverse problem of identifying the bounded 
operators connected with $S \oplus S^*$ through the relevant 
skew-commutation relation in Theorem \ref{skew is S+S*} and \ref{self adjoint S+S*}. 
   
    \section{Skew Commutators}\label{section skew commutators}
    The proofs that we present in this section are using the techniques from \cite{CG}.
    Let us look at the very important Lemma from \cite{CG}, which we will be using throughout this section.
    \begin{lemma}[\cite{CG}]\label{LGU} 
        If $H_1$ and $H_2$ are two Hankel operators and $T_1$ and $T_2$ are two Toeplitz operators on $\ck$, then 
        \begin{equation*}
            S^*(T_1H_1-H_2T_2)-(T_1H_1-H_2T_2)S=S^*T_1\mathcal{D}^2_{S^*} H_1+H_2\mathcal{D}^2_{S^*} T_2S.
        \end{equation*}
    \end{lemma}
    \subsection{When is a Toeplitz operator skew commutator of a Hankel Operator?}
    
    \begin{proposition}\label{prop1}
        Let $T$ be a Toeplitz operator and $H$ be a Hankel operator
        on $\ck$. Then $TH^* = HT$ if and only if the following conditions are satisfied:
        \begin{enumerate}
            \item  $S^*T\mathcal{D}^2_{S^*} H^*+H\mathcal{D}^2_{S^*} TS=0;$\label{1.1}
            \item $(TH^*-HT)\mathcal{D}^2_{S^*}=0.$\label{2.1}
        \end{enumerate}
        \end{proposition}
        \begin{proof}
            Since $H$ is a Hankel operator, $H^*$ is also a Hankel operator. Assume that $TH^*=HT$ on $\ck.$ Statement (\ref{2.1}) is immediate. Setting $T_1=T_2=T$, $H_1=H^*$ and $H_2=H$ in Lemma \ref{LGU}, statement (\ref{1.1}) follows.
            Conversely, from statement (\ref{1.1}) and Lemma \ref{LGU} we get
            \begin{align}\label{extra}
                S^*(TH^*-HT)-(TH^*-HT)S=0.
            \end{align}
             Using statement (\ref{2.1}), we get
            \begin{align*}
                TH^*-HT=& (TH^*-HT)SS^*\\
                \overset{\eqref{extra}}{=}& S^*(TH^*-HT)S^*.
            \end{align*}
            Applying this recursively, we get
            \begin{align*}
                 TH^*-HT=&{S^*}^n(TH^*-HT){S^*}^n, \text{ for } n \in \mathbb{N}.
            \end{align*}
            \begin{align*}
                ||(TH^*-HT)x||&=||{S^*}^n(TH^*-HT){S^*}^nx||\\
                &\leq||(TH^*-HT)|| ||{S^*}^nx||.
            \end{align*}
            As ${S^*}^n \xrightarrow{\text{SOT}} 0,$ we get $TH^*=HT$.
            \end{proof}
    
    \begin{definition}[\cite{CG}]
        Let $F$ be $m\times m$ matrix with entries from $L^2.$ The matrix valued function $F$ is said to be regular if $m$ column vectors of $F$ are linearly independent in $L^2_m.$
    \end{definition}
    The unilateral shift $S$ on $H^2_n$ is nothing but $T_{zI}$ where $I$ denotes the $n \times n$
    identity matrix. 
    We will denote the operator $T_{zI}$ simply by $T_z$.
    The wandering subspace for $T_{z}$ is $\mathbb{C}^n.$
    Let $\{e_1,e_2,\dots,e_n\}$ be the standard basis for the wandering subspace. Thus
            \begin{equation*}
                \mathcal{D}^2_{T^*_z}=I-T_zT_z^*=\sum_{i=1}^{n} e_i\otimes e_i.
            \end{equation*}
    For $x,y \in \mathcal{H},$ the rank one operator $ x\otimes y$ on $\mathcal{H}$ is defined as 
    $$
        (x\otimes y) (z)=\langle z, y\rangle x.
    $$
    We will be using the following well known lemma throughout this section
    \begin{lemma}\label{Tensor lemma}
       Let $x_i,y_i,z_i,w_i \in \mathcal{H},$ for $1 \leq i\leq n.$ Suppose that $\{y_i: 1\leq i \leq n\}$ is linearly independent and
        $$
            \sum_{i=1}^{n}(x_i \otimes y_i)=\sum_{i=1}^{n} (z_i \otimes w_i).
        $$
        Then for each $j$, there exist scalars $a_{ij},~ 1 \leq i \leq n$, such that 
        $$
            x_j=\sum_{i=1}^{n} a_{ij}z_i.
        $$
     \end{lemma}
    
    \begin{theorem}\label{Th1}
        Assume that both $P(\overline z\phi(z))$ and $P(\overline z\phi^*(z))$ are regular. Then $T_\phi H_\psi^* =H_\psi T_\phi$  if and only if there exists $A\in M_n(\mathbb{C})$ such that the following holds:
        \begin{enumerate}[(i)]
            \item $\psi(z)-\phi(\overline z)A$ is analytic;\label{1i}
            \item $\tilde\psi(z)+A\phi(z)$ is analytic;\label{1ii}
            \item $\phi(\overline z)A\phi(z)$ is analytic.\label{1iii}
        \end{enumerate}
    \end{theorem}
        \begin{proof}
        Let $A$ be any matrix of order $n$. Let $\phi \in \li(M_n(\mathbb{C}))$ such that
        $$
        \phi(z)=\sum_{-\infty}^{\infty} z^n\phi_n.
         $$
        Note that 
            \begin{align*}
                A(P\phi)(\overline z)&= \sum_{n=0}^{\infty} \overline{z} ^n A\phi_n\\
                &=A\phi(0)+(I-P)(A\phi(\overline z))
             \end{align*}
        and
             $$
            zP(\overline z A\phi(\overline z))=P(A\phi(\overline z)) -A\phi(0).
            $$
        Hence we get
        \begin{equation}\label{symbol-1}
            zP(\overline z A\phi(\overline z))+A(P\phi)(\overline z)=A\phi(\overline z).
        \end{equation}
        
        Let (\ref{1i}), (\ref{1ii}) and (\ref{1iii}) holds for some matrix $A$ of order $n.$ Using statement (\ref{1i}) and (\ref{1ii}) it follows that
        \begin{align*}
                T_{\phi(z)} H^*_{\psi(z)}-H_{\psi(z)} T_{\phi(z)}&=T_{\phi(z)} H_{\tilde \psi(z)}-H_{\phi(\overline z)A}T_{\phi(z)}\\
                &=T_{\phi(z)} H_{-A\phi(z)}-H_{\phi(\overline z)A}T_{\phi(z)}.
        \end{align*}
         This gives
        \begin{align*}
                [(T_{\phi(z)} H^*_{\psi(z)}-H_{\psi(z)} T_{\phi(z)})e_k]_{k=1}^{n}&=
               P(\phi(z)P(-\overline z A \phi (\overline z)))-P(\overline{z}\phi(z)A(P\phi)(\overline z))\\
                &=P[-\overline z\phi(z)(zP(\overline z A\phi(\overline z))+A(P\phi)(\overline z))].
        \end{align*}
        Using equation (\ref{symbol-1}) we get
        \begin{equation}\label{Symbol comparison 1}
                [(T_{\phi(z)} H^*_{\psi(z)}-H_{\psi(z)} T_{\phi(z)})e_k]_{k=1}^{n}=P(-\overline z\phi(z)A\phi(\overline z)).
        \end{equation}
         By statement (\ref{1iii}) we have 
        $$ 
            (T_{\phi(z)} H^*_{\psi(z)}-H_{\psi(z)} T_{\phi(z)})e_k=0, 
        $$
        for every $k=1,\dots,n.$ Next by the virtue of Proposition \ref{prop1} it is enough to show statement (\ref{1.1}). By statements (\ref{1i}) and (\ref{1ii}) we have 
        \begin{align*}
            P(\overline z (\psi(\overline z)-\phi(z)A))=0
        \end{align*}
        and 
        \begin{align*}
            P(\overline z (\psi(\overline z)+\phi^*(z)A^*))=0,
        \end{align*}
        respectively. Therefore
        \begin{align*}
            H_\psi e_j = \sum_{i=1}^{n}a_{ij}T_z^*T_\phi e_i =-\sum_{i=1}^{n}\overline{a_{ji}}T^*_zT^*_\phi e_i.
        \end{align*}
        This gives 
        \begin{align*}
            H_\psi D^2_{{T_z}^*} T_\phi T_z &= H_\psi (\sum_{i=1}^{n} e_i\otimes e_i) {T_\phi}T_z\\
            &=\sum_{i=1}^{n}H_\psi e_i \otimes T_z^*T^*_\phi e_i\\
            &=\sum_{i=1}^{n}(\sum_{j=1}^{n}a_{ji}{T_z}^*T_\phi e_j) \otimes T_z^*T^*_\phi e_i\\
            &=\sum_{j=1}^{n}\sum_{i=1}^{n}{T_z}^*T_\phi e_j \otimes \overline{a_{ji}}T_z^*T^*_\phi e_i\\
            &=\sum_{j=1}^{n}({T_z}^*T_\phi e_j \otimes \sum_{i=1}^{n} \overline{a_{ji}}T_z^*T^*_\phi e_i)\\
            &=-\sum_{j=1}^{n}{T_z}^*T_\phi e_j \otimes H_\psi e_j\\
            &=-T_z^*T_\phi (\sum_{i=1}^{n} e_i\otimes e_i) H_\psi^*\\
            &=-T_z^*T_\phi D^2_{T_z*} H_\psi^*.
        \end{align*}
        Conversely, let us assume that $T_\phi H^*_\psi =H_\psi T_\phi$. Using Proposition \ref{prop1} we have 
           \begin{align}\label{eq1}
             \sum_{i=1}^{n} (T_z^*T_\phi e_i \otimes H_\psi e_i + H_\psi e_i \otimes T^*_zT^*_\phi e_i)=0.
           \end{align}
            Since $P(\overline z \phi^*(z))$ is regular, $\{T^*_zT^*_\phi e_i\}_{i=1}^{n}$ is linearly independent. By Lemma \ref{Tensor lemma} for each $1\leq j\leq n,$
            there exist scalars $a_{ij}$ such that
            \begin{align}
                H_\psi e_j = \sum_{i=1}^{n}a_{ij}T_z^*T_\phi e_i
                =T^*_z\sum_{i=1}^{n}a_{ij}T_\phi e_i\label{eqo}.
            \end{align}
            So we get
            \begin{align*}
                [H_{\psi(z)}e_i]_{i=1}^{n}=P(\overline z \phi(z)A),
             \end{align*}
             Equivalently
             \begin{align*}
                 H_{\psi(z)}=H_{\phi(\overline z)A}.
             \end{align*}
            This gives $\psi(\overline z)-\phi(z)A$ is conjugate analytic, equivalently $\psi(z)-\phi(\overline z)A$ is analytic.
            Now substituting equation (\ref{eqo}) in equation (\ref{eq1}), we get
            \begin{align*}
            -\sum_{j=1}^{n} (T_z^*T_\phi e_j \otimes H_\psi e_j )=\sum_{j=1}^{n}\sum_{i=1}^{n}(a_{ij}T_z^*T_\phi e_i \otimes T^*_zT^*_\phi e_j)\\
            = \sum_{j=1}^{n}\sum_{i=1}^{n}(T_z^*T^*_\phi e_i \otimes \overline{a_{ij}}T^*_zT^*_\phi e_j)\\
            = \sum_{i=1}^{n}(T_z^*T_\phi e_i \otimes \sum_{j=1}^{n}\overline{a_{ij}}T^*_zT^*_\phi e_j)\\
             = \sum_{j=1}^{n}(T_z^*T_\phi e_j \otimes \sum_{i=1}^{n}\overline{a_{ji}}T^*_zT^*_\phi e_i).
            \end{align*}
            Since $P(\overline z \phi(z))$ is regular, for each $1 \leq j \leq n$
            \begin{equation*}
                 H_\psi e_j=-\sum_{i=1}^{n}\overline{a_{ji}}T^*_zT^*_\phi e_i\\
                 =-T^*_z\sum_{i=1}^{n}\overline{a_{ji}}T^*_\phi e_i.\\
            \end{equation*}
            In terms of symbol, it gives
            \begin{equation*}
                 P(\overline z\psi(\overline z))=[H_{\psi(z)}e_j]_{j=1}^{n}= -P(\overline z \phi^*(z)A^*).
            \end{equation*}
            This is equivalent to say $\psi(\overline z)+\phi^*(z)A^*$ is conjugate analytic. In other words, $ \tilde \psi (z)+A\phi(z)$ is analytic. We have
            \begin{align*}
                0&=T_\phi H_\psi^*-H_\psi T_\phi =T_\phi H_{\tilde{\psi}(z)}-H_{\phi(\overline z)A}T_{\phi(z)}=-T_\phi H_{A\phi(z)}-H_{\phi(\overline z)A}T_{\phi(z)}.
            \end{align*}
            Using equation (\ref{Symbol comparison 1}) we get $P(\overline z\phi(z)A\phi(\overline{z}))=0$. It follows that  $\phi(z)A\phi(\overline{z})$ is conjugate analytic. Equivalently $\phi(\overline{z})A\phi(z)$ is analytic.
    \end{proof} 
    \subsection{When is a Hankel operator skew commutator of a Toeplitz operator?}
    \begin{proposition}\label{prop2}
        Let $T$ be a Toeplitz operator and $H$ be a Hankel operator on $\ck$. Then $TH=HT^*$ if and only if the following conditions are satisfied:
        \begin{enumerate}
            \item $S^*T\mathcal{D}^2_{S^*} H+H\mathcal{D}^2_{S^*} T^*S=0;$\label{1.2}
            \item $(TH-HT^*)\mathcal{D}^2_{S^*}=0.$\label{2.2}
        \end{enumerate}
        \begin{proof}
             Since $T$ is Toeplitz, $T^*$ is also a Toeplitz operator. Assume that $TH=HT^*$ on $\ck$. Statement (\ref{2.2}) is immediate. Setting $T_1=T$, $T_2=T^*$, and  $H_1=H_2=H$ in Lemma \ref{LGU}, we get the statement (\ref{1.2}). Conversely using the statements (\ref{1.2}), (\ref{2.2}) and Lemma \ref{LGU} we have
             \begin{align*}
                 TH-HT^*=S^*(TH-HT^*)S^*.
             \end{align*}
            The similar argument as given in the Proposition \ref{prop1} completes the proof.
        \end{proof}
    \end{proposition}
    
    \begin{theorem}\label{Th2}
        Assume that $P(\overline z \phi(z))$ is regular. Then $T_\phi H_\psi =H_\psi T_\phi^*$  if and only if there exist $A\in M_n(\mathbb{C})$ such that the following holds:
        \begin{enumerate}[(i)]
            \item $\psi(z)-\phi ( \overline z)A$ is analytic;\label{2i}
            \item $\psi(z)+A\phi^*(z)$ is analytic;\label{2ii}
            \item $\phi(\overline z)A\phi(z)^*$ is analytic.\label{2iii}
       \end{enumerate}
       \end{theorem}
        \begin{proof}
             Let us first assume that the statements (\ref{2i}),(\ref{2ii}) and (\ref{2iii}) hold for some matrix $A$ of order $n.$ By the virtue of Proposition \ref{prop2} we first show
             $$
             [(T_\phi H_\psi-H_\psi T^*_\phi)e_i]_{i=1}^{n}=0.
             $$
            Using statements (\ref{2i}) and (\ref{2ii}) we have
            \begin{align*}
                T_\phi H_\psi-H_\psi T^*_\phi&=T_\phi H_{\phi(\overline z)A}-H_{\phi(\overline z)A}T_{\phi^*(z)}.
                \end{align*}
                Also,
                \begin{align*}
                P(\phi(z)P(\overline z \phi(z)A))-P(\overline{z}\phi(z)A(P\phi^*)(\overline z))
                &=P(\overline z\phi(z)[zP(\overline z \phi(z)A)-A(P\phi^*)(\overline z)])\\
                &=-P(\overline z\phi(z)A\phi^*(\overline{z})).
            \end{align*}
             This gives
            \begin{equation}\label{symbol comparision 2}
                 [(T_\phi H_\psi-H_\psi T^*_\phi)e_i]_{i=1}^{n}= -P(\overline z\phi(z)A\phi^*(\overline{z}))=0
            \end{equation}
            The last equality is due to statement (\ref{2iii}). The statements (\ref{2i}) and (\ref{2ii}) gives
            \begin{align*}
                P(\overline z (\psi(\overline z)-\phi(z)A))=0
            \end{align*}
            and 
            \begin{align*}
                P(\overline z (\psi^*( z)+\phi(z)A^*))=0,
            \end{align*}
            respectively. Thus
            \begin{align*}
                H_\psi e_j = \sum_{i=1}^{n}a_{ij}T_z^*T_\phi e_i, ~H^*_\psi e_j=-\sum_{i=1}^{n}\overline{a_{ji}}T^*_zT_\phi e_i.
            \end{align*}
            Using the similar arguments as in Theorem \ref{Th1} we have
            \begin{align*}
               H_\psi \mathcal{D}^2_{T_z^*}T^*_\phi T_z&=\sum_{i=1}^{n} H_\psi e_i \otimes T^*_zT_\phi e_i=-T_z^*T_\phi \mathcal{D}^2_{T_z^*}H_\psi.
            \end{align*}
             Conversely, let us assume that $T_\phi H_\psi =H_\psi T^*_\phi$.
            By Proposition \ref{prop2} we have 
           \begin{align} 
               0=&T_z^*T_\phi (\sum_{i=1}^{n} e_i\otimes e_i) H_\psi +H_\psi (\sum_{i=1}^{n} e_i\otimes e_i) {T^*_\phi}T_z \nonumber\\
              =&\sum_{i=1}^{n} (T_z^*T_\phi e_i \otimes H^*_\psi e_i + H_\psi e_i \otimes T^*_zT_\phi e_i)\label{eq2}
           \end{align}
            Since $P(\overline z \phi(z))$ is regular, $\{T^*_zT_\phi e_i\}_{i=1}^{n}$ is linearly independent. By Lemma \ref{Tensor lemma} for $1\leq j \leq n$ there exist scalars $a_{ij}$ such that
            \begin{align*}
                H_\psi e_j = \sum_{i=1}^{n}a_{ij}T_z^*T_\phi e_i
                =T^*_z\sum_{i=1}^{n}a_{ij}T_\phi e_i.
            \end{align*}
            Therefore $H_\psi=H_{\phi(\overline{z})A}$. It gives $\psi(\overline{z})-\phi(z)A$ is conjugate analytic. Equivalently $\psi(z)-\phi(\overline z)A$ is analytic. Also,
            \begin{align*}
            \sum_{i=1}^{n} (H_\psi e_i \otimes T_z^*T_\phi e_i )=\sum_{i=1}^{n}(\sum_{j=1}^{n}a_{ji}T_z^*T_\phi e_j) \otimes T^*_zT_\phi e_i\\
            = \sum_{i=1}^{n}(\sum_{j=1}^{n}T_z^*T_\phi e_j \otimes \overline{a_{ji}}T^*_zT_\phi e_i)\\
            = \sum_{j=1}^{n}(T_z^*T_\phi e_j \otimes \sum_{i=1}^{n}\overline{a_{ji}}T^*_zT_\phi e_i)\\
             = \sum_{i=1}^{n}(T_z^*T_\phi e_i \otimes \sum_{j=1}^{n}\overline{a_{ij}}T^*_zT_\phi e_j).
            \end{align*}
            Comparing it with equation (\ref{eq2}), we have
            \begin{equation*}
                \sum_{i=1}^{n} (T_z^*T_\phi e_i \otimes H^*_\psi e_i)=-\sum_{i=1}^{n}(T_z^*T_\phi e_i \otimes \sum_{j=1}^{n}\overline{a_{ij}}T^*_zT_\phi e_j).
            \end{equation*}
            Since $P(\overline z \phi(z))$ is regular, for each $1\leq i\leq n$
            \begin{equation*}
                 H^*_\psi e_i=-\sum_{j=1}^{n}\overline{a_{ij}}T^*_zT_\phi e_j.\\
            \end{equation*}
            Writing it as operators, we get
                 $H_\psi^*=H_{-\phi(\overline z)A^*}$
            It is equivalent to say $\tilde \psi(\overline z)+\phi(z)A^*$ is conjugate analytic. In other words, $\psi(z)+A\phi^*(z)$ is analytic. Next by equation (\ref{symbol comparision 2}) we get $P(\overline z\phi(z)A\phi^*(\overline{z}))=0.$ This gives $\phi(z)A\phi^*(\overline{z})$ is conjugate analytic, in other words $\phi(\overline z)A\phi^*(z)$ is analytic.
        \end{proof}
    
    \subsection{When is the product of Hankel and Toeplitz operators self-adjoint?}
    \begin{proposition}\label{prop3}
        Let $T$ be a Toeplitz operator, and let $H$ be a Hankel operator on $\ck$. Then $HT=T^*H^*$ if and only if the following conditions are satisfied:
        \begin{enumerate}
            \item $S^*T^*\mathcal{D}^2_{S^*} H^*+H\mathcal{D}^2_{S^*} TS=0;$\label{1.3}
            \item $(HT-T^*H^*)\mathcal{D}^2_{S^*}=0.$\label{2.3}
        \end{enumerate}
        \end{proposition}
        \begin{proof}
             Note that $H^*$ is a Hankel operator and $T^*$ is a Toeplitz operator. Assume that $HT=T^*H^*$ on $\ck$. The statement (\ref{2.3}) follows immediately. Setting $T_1=T^*$, $T_2=T$, $H_1=H^*$ and $H_2=H$ in Lemma \ref{LGU}, we get the statement (\ref{1.3}). The converse part is similar to the arguments given in the Proposition \ref{prop1}. Using statements (\ref{1.3}) and (\ref{2.3}), Lemma \ref{LGU}, and applying the same argument recursively, we get
             $$
                HT-T^*H^*={S^*}^n(HT-T^*H^*){S^*}^n,
             $$
             for $n \in \mathbb{N}.$ The fact that ${S^*}^n \xrightarrow{\text{SOT}} 0$ gives $HT=T^*H^*.$
        \end{proof}
    \begin{theorem}\label{Th3}
        Assume that $P(\overline z \phi^*(z))$ is regular. Then $H_\psi T_\phi =T_\phi^*H_\psi^*$  if and only if there exist $A\in M_n(\mathbb{C})$ such that the following hold:
        \begin{enumerate}[(i)]
            \item $\psi(z)+\tilde\phi(z)A^*$ is analytic;\label{3i}
            \item $\tilde\psi(z)-A^*\phi(z)$ is analytic;\label{3ii}
            \item $\tilde \phi(z)A\phi(z)$ is analytic.\label{3iii}
        \end{enumerate}
        \end{theorem}
        \begin{proof}
            Note that 
            \begin{align*}
                [(-H_{\tilde \phi(z)A^*}T_{\phi(z)}-T_{\phi^*(z)}H_{A^*\phi(z)})e_i]_{i=1}^{n}&=P(-\overline z\tilde \phi(\overline z)A^*(P\phi)(\overline z)-P(\phi^*(z)P(\overline z A^* \phi(\overline z)))\\
                &=P(\overline z\phi^*(z)(-A^*(P\phi)(\overline z)-zP(\overline z A^* \phi(\overline z))))\\
                &=-P(\overline z\phi^*(z)A^*\phi(\overline{z})).
            \end{align*}
            Let us assume that statements (\ref{3i}),(\ref{3ii}) and (\ref{3iii}) hold for some matrix $A$ of order $n.$
            \begin{align*}
                H_\psi T_\phi -T^*_\phi H_\psi^*
               &=H_\psi T_\phi-T^*_\phi H_{\tilde \psi}\\
                &=-H_{\tilde \phi(z)A^*}T_{\phi(z)}-T_{\phi^*(z)}H_{A^*\phi(z)}
            \end{align*}
            The last equality is due to the statements (\ref{3i}) and (\ref{3ii}). This gives
            \begin{align}\label{symbol comparison 3}
                [(H_\psi T_\phi -T^*_\phi H_\psi^*)e_i]_{i=1}^{n}
                &=-P(\overline z\phi^*(z)A^*\phi(\overline{z})).
            \end{align}
            By statement (\ref{3iii}) we have
            $P(\overline z\phi^*(z)A^*\phi(\overline{z}))=0.$ The statements (\ref{3i}) and (\ref{3ii}) gives
            \begin{align*}
                H_\psi e_i = \sum_{j=1}^{n}a_{ji}T_z^*T^*_\phi e_j =-\sum_{j=1}^{n}\overline{a_{ij}}T^*_zT^*_\phi e_j
            \end{align*}
            Thus using the similar arguments in Theorem \ref{Th1} we have
            \begin{align*}
                {T_z^*} T_\phi^*\mathcal{D}^2_{T_z^*} H_\psi^*+H_\psi \mathcal{D}^2_{T_z^*} T_\phi T_z=0
            \end{align*}
            Conversely, let us assume that $H_\psi T_\phi =T_\phi^*H_\psi^*$.
            By Proposition \ref{prop3} we have 
           \begin{align} 
               0=&T_z^*T_\phi^* (\sum_{i=1}^{n} e_i\otimes e_i) H_\psi^* +H_\psi (\sum_{i=1}^{n} e_i\otimes e_i) {T_\phi}T_z \nonumber\\
              =&\sum_{i=1}^{n} (T_z^*T^*_\phi e_i \otimes H_\psi e_i + H_\psi e_i \otimes T^*_zT^*_\phi e_i)\label{eq3}
           \end{align}
            Since $P(\overline z \phi^*(z))$ is regular, $\{T^*_zT^*_\phi e_i\}_{i=1}^{n}$ is linearly independent.
            By Lemma \ref{Tensor lemma} for $1\leq i\leq n$, there exist $n$ scalars $a_{ji}$ such that
            \begin{align}
                H_\psi e_i = \sum_{j=1}^{n}a_{ji}T_z^*T^*_\phi e_j
                =T^*_zT_{\phi^* A} e_i.
            \end{align}
            Thus $P(\overline z \psi(\overline z))=P(\overline z \phi^*(z)A).$ This gives $\psi(\overline z)-\phi^*(z)A$ is co-analytic. In other words statement (\ref{3ii}) holds.
            Using equation (\ref{eq3}), we get
            \begin{align*}
            \sum_{i=1}^{n} (H_\psi e_i \otimes T^*_zT^*_\phi e_i)=\sum_{i=1}^{n}(\sum_{j=1}^{n}a_{ji}T_z^*T^*_\phi e_j) \otimes T^*_zT^*_\phi e_i\\
            = \sum_{i=1}^{n}(\sum_{j=1}^{n}T_z^*T^*_\phi e_j \otimes \overline{a_{ji}}T^*_zT^*_\phi e_i)\\
            = \sum_{j=1}^{n}(T_z^*T^*_\phi e_j \otimes \sum_{i=1}^{n}\overline{a_{ji}}T^*_zT^*_\phi e_i)\\
             = \sum_{i=1}^{n}(T_z^*T^*_\phi e_i \otimes \sum_{j=1}^{n}\overline{a_{ij}}T^*_zT^*_\phi e_j).
            \end{align*}
             Comparing with equation (\ref{eq3}), we have
            \begin{equation*}
                \sum_{i=1}^{n} (T_z^*T^*_\phi e_i \otimes H_\psi e_i)=-\sum_{i=1}^{n}(T_z^*T^*_\phi e_i \otimes \sum_{j=1}^{n}\overline{a_{ij}}T^*_zT^*_\phi e_j).
            \end{equation*}
            Also, for each $1\leq i\leq n$
            \begin{equation*}
                 H_\psi e_i=-\sum_{j=1}^{n}\overline{a_{ij}}T^*_zT^*_\phi e_j.
            \end{equation*}
            Therefore $P(\overline z \psi(\overline z))=-P(\overline z \phi^*(z)A^*)$
            It is equivalent to say $\psi(\overline z)+\phi^*(z)A^*$ is conjugate analytic. In other words, $ \psi(z)+\tilde \phi(z)A^*$ is analytic. The statement (\ref{3iii}) follows  from equation (\ref{symbol comparison 3}).
        \end{proof}
    \subsection{When is a Toeplitz operator skew commutator of a Toeplitz operator?}
    \begin{proposition}
        Let $T_i$, for $i=1,2,3,4$, be Toeplitz operators on Hilbert space $\ck$. Then $T_1T_2=T_3T_4$ if and only if the following conditions are satisfied:
        \begin{enumerate}
            \item $S^*T_1\mathcal{D}^2_{S^*}T_2S-S^*T_3\mathcal{D}^2_{S^*}T_4S=0;$
            \item $(T_1T_2-T_3T_4)\mathcal{D}^2_{S^*}=0;$
            \item $(T_1T_2-T_3T_4)^*\mathcal{D}^2_{S^*}=0.$
        \end{enumerate}
    \end{proposition}
        \begin{proof}
            Note that
            \begin{align*}
                S^*T_1T_2S&=S^*T_1SS^*T_2S+S^*T_1\mathcal{D}^2_{S^*}T_2S\\
                &=T_1T_2+S^*T_1\mathcal{D}^2_{S^*}T_2S,\\
                S^*T_3T_4S&=S^*T_3SS^*T_4S+S^*T_3\mathcal{D}^2_{S^*}T_4S\\
                &=T_3T_4+S^*T_3\mathcal{D}^2_{S^*}T_4S.
            \end{align*}
            Both together give
            \begin{equation}\label{t1t2-t3t4}
                S^*T_1T_2S-S^*T_3T_4S=T_1T_2-T_3T_4+S^*T_1\mathcal{D}^2_{S^*}T_2S-S^*T_3\mathcal{D}^2_{S^*}T_4S.
            \end{equation}
            Assume that $T_1T_2=T_3T_4$, then by above statements (\ref{t1}),(\ref{t2}) and (\ref{t3}) are immediate. 
            Let $\Delta:=T_1T_2-T_3T_4,$ and suppose statement (\ref{t1}) holds. Then equation (\ref{t1t2-t3t4}) can be written as 
            \begin{align*}
                \Delta=S^*\Delta S
            \end{align*}
            If statements (\ref{t2}) and (\ref{t3}) hold, then we have
            \begin{align*}
                \Delta =\Delta SS^* \text{ and }\Delta =SS^* \Delta.
            \end{align*}
            This gives
            \begin{align*}
            \Delta=&SS^*\Delta SS^*=S\Delta S^*
            \end{align*}
        Applying recursively, we get
        \begin{align*}
            \Delta=S^n\Delta {S^*}^n, n \in \mathbb{N}.
        \end{align*}
        Since ${S^*}^n \xrightarrow{SOT} 0,$ we get $T_1T_2=T_3T_4.$
        \end{proof}
    \begin{corollary}\label{T1T2T2T1*}
        Let $T_1$ and $T_2$ be Toeplitz operators on a Hilbert space $\ck.$ Then $T_1T_2=T_2T_1^*$ if and only if the following conditions are satisfied:
        \begin{enumerate}
            \item $S^*T_1\mathcal{D}^2_{S^*}T_2S-S^*T_2\mathcal{D}^2_{S^*}T_1^*S=0;$\label{t1}
            \item $(T_1T_2-T_2T_1^*)\mathcal{D}^2_{S^*}=0;$\label{t2}
            \item $(T_1T_2-T_2T_1^*)^*\mathcal{D}^2_{S^*}=0.$\label{t3}
        \end{enumerate}
    \end{corollary}
    \begin{remark}
        Let $T_1$ and $T_2$ be Toeplitz operators. Then $T_1T_2$ is a Toeplitz operator if and only if $S^*T_1\mathcal{D}^2_{S^*}T_2S=0.$
     \end{remark}
        The proof of the remark follows from the following identity: 
        \begin{equation*}
            S^*T_1T_2S-T_1T_2=S^*T_1\mathcal{D}^2_{S^*}T_2S+S^*T_1SS^*T_2S-T_1T_2=S^*T_1\mathcal{D}^2_{S^*}T_2S.
        \end{equation*}
    
    \begin{theorem}\label{Th4}
        Assume that $P(\overline{z}\psi(z))$ and $P(\overline{z}{\psi^*}(z))$ are regular.
        Then $T_\phi T_\psi =T_\psi T^*_\phi$ if and only if there exist $A\in M_n(\mathbb{C})$ such that  the following holds:
        \begin{enumerate}[(i)]
            \item $\phi(z)-\psi(z)A$ is co-analytic;\label{4i}
            \item $\phi(z)-\psi^*(z)A^*$ is co-analytic;\label{4ii}
            \item $\phi(z) \psi(z)=\psi(z) \phi^*(z).$ \label{4iii}
        \end{enumerate}
    \end{theorem}
        \begin{proof}
            Let us first assume that the statements (\ref{4i}), (\ref{4ii}) and (\ref{4iii}) hold for some matrix $A$ of order $n.$ Then using statement (\ref{4iii}) we have 
            \begin{align*}
                P(\phi \psi- \psi \phi^*)=0.
            \end{align*}
            This implies
            \begin{align*}
                P[(\phi-\psi A)\psi-\psi(\phi^*-A\psi)]=0.
            \end{align*}
            By statement (\ref{4i})
            \begin{align*}
                P[(\phi-\psi A)P(\psi)]-P[\psi(\phi^*-A\psi)]=0.
            \end{align*}
            Using statement (\ref{4ii}) we have
            \begin{align*}
                0=&P[(\phi-\psi A)P(\psi)]-P[\psi P(\phi^*-A\psi)]\\
                =&P[\phi P(\psi) -\psi AP(\psi)-\psi P(\phi^*)+\psi AP(\psi)]\\
                =&P[\phi P(\psi)]-P[\psi P(\phi^*)].
            \end{align*}
            This proves that
            \begin{align*}
                [(T_\phi T_\psi-T_\psi T^*_\phi)e_i]_{i=1}^{n}=0.
            \end{align*}
            Note that
            by statement (\ref{4i}) we have
            \begin{align*}
                (I-P)(\phi^*-A^*\psi^*)=0.
            \end{align*}
            Thus
             \begin{align*}
                P[\psi^*(I-P)\phi^* - \phi(I-P)\psi^*]
                =P[(\psi^*A^*-\phi)(I-P)\psi^*]=0
            \end{align*}
            The last equality is due to statement (\ref{4ii}).
            Therefore
            \begin{align*}
                P(\psi^*\phi^* - \phi\psi^*)&=P(\psi^*P(\phi^*) - \phi P(\psi^*)).
            \end{align*}
            By taking adjoint of the statement (\ref{4iii}), we have
            \begin{align*}
                P(\psi^*\phi^* - \phi\psi^*)=0.
            \end{align*}
            Therefore we have
            \begin{align*}
                [(T^*_\psi T^*_\phi -T_\phi T^*_\psi)e_i]_{i=1}^{n}=0.
            \end{align*}
            The statements (\ref{4i}) and (\ref{4ii}) gives
            \begin{align*}
                P(\overline z(\phi-\psi A)) =0
            \end{align*}
            and
            \begin{align*}
                P(\overline z(\phi-\psi^*A^*)) =0,
            \end{align*}
            respectively.
            Thus
            \begin{align*}
                S^*T_\phi e_i = \sum_{j=1}^{n}a_{ji}S^*T_\psi e_j= \sum_{j=1}^{n}\overline{a_{ij}}S^*T_\psi^* e_j.
            \end{align*}
            Using the similar arguments in Theorem \ref{Th1}, we get
            \begin{align*}
                T_z^*T_\phi \mathcal{D}^2_{T_z^*}T_\psi T_z=T_z^*T_\psi \mathcal{D}^2_{T_z^*}T_\phi^* T_z.
            \end{align*}
            Hence by Corollary \ref{T1T2T2T1*} we are through. Conversely, let us assume that $T_\phi T_\psi =T_\psi T^*_\phi.$ Using statement (\ref{t1}) of Corollary \ref{T1T2T2T1*} we have
            \begin{align*}
               0=& T_z^*T_\phi \mathcal{D}^2_{T_z^*}T_\psi T_z-T_z^*T_\psi \mathcal{D}^2_{T_z^*}T_\phi^* T_z\\
               =&T_z^*T_\phi (\sum_{i=1}^{n} e_i\otimes e_i)T_\psi T_z-T_z^*T_\psi (\sum_{i=1}^{n} e_i\otimes e_i)T_\phi^*T_z.
            \end{align*}
            This implies
             \begin{align*}  
                \sum_{i=1}^{n} (T_z^*T_\phi e_i \otimes T_z^*T^*_\psi  e_i) = \sum_{i=1}^{n} (T_z^*T_\psi e_i \otimes T_z^*T_\phi e_i).
            \end{align*}
            Since $P(\overline{z}\psi^*(z))$ is regular, for $1\leq i\leq n,$ there exist scalars $a_{ji}$ such that
            \begin{align}
                S^*T_\phi e_i = \sum_{j=1}^{n}a_{ji}S^*T_\psi e_j.\label{eq4}
            \end{align}
            This gives $P(\overline{z}\phi(z))=P(\overline{z}\psi(z)A),$ Equivalently statement (\ref{4i}) holds.
            By equation (\ref{eq4}) we get
            \begin{align*}
            \sum_{i=1}^{n} (S^*T_\phi e_i \otimes S^*T^*_\psi  e_i)=&\sum_{i=1}^{n} (\sum_{j=1}^{n}a_{ji}S^*T_\psi e_j \otimes S^*T^*_\psi  e_i)\\
            =&\sum_{j=1}^{n}(S^*T_\psi e_j \otimes \sum_{i=1}^{n} \overline{a_{ij}}S^*T^*_\psi  e_i)\\
            =&  \sum_{i=1}^{n} (S^*T_\psi e_i \otimes \sum_{j=1}^{n}\overline{a_{ij}}S^*T_\psi^* e_j)
            \end{align*}
            Since $P(\overline{z}\psi(z))$ is regular, we have
            \begin{align*}
                S^*T_\phi e_i=&\sum_{j=1}^{n}\overline{a_{ij}}S^*T_\psi^* e_j.
            \end{align*}
            This gives $P(\overline{z}\phi(z))=P(\overline{z}\psi^*(z)A^*).$ Equivalently statement (\ref{4ii}) holds. Using the arguments given due to statements (\ref{4i}) and (\ref{4ii}) we note that
            \begin{align*}
                P(\phi P(\psi))-P(\psi P(\phi^*))=P(\phi \psi- \psi \phi^*),
            \end{align*}
            and
            \begin{align*}
                P(\psi^*P(\phi^*) - \phi P(\psi^*))=P(\psi^*\phi^* - \phi\psi^*).
            \end{align*}
            Thus by Corollary \ref{T1T2T2T1*} we have
                \begin{align*}
                    P(\phi \psi- \psi \phi^*)=0=P(\psi^*\phi^* - \phi\psi^*).
                \end{align*}
               Hence $\phi \psi= \psi \phi^*.$
            
        \end{proof}
    
    \subsection{When is a Hankel operator skew commutator of a Hankel operator?}
    \begin{theorem}[Theorem 4.3 of \cite{CG}]
        Assume $\phi_1$ and $\phi_2$ are regular. Then $H_{\phi_1}H_{\psi_1}=H_{\psi_2}H_{\phi_2}$ if and only if there exist $A \in M_n(\mathbb{C})$ such that the following hold:
        \begin{enumerate}
            \item $\psi_2(z)=\phi_1(z)A$;
            \item $\psi_1(z)=A\phi_2(z).$
        \end{enumerate}
    \end{theorem}
    \begin{corollary}
         Assume that $\phi(z)$ is regular. Then, $H_\phi H_\psi=H^*_\psi H_\phi$ if and only if $\psi(z)=A\phi(z)$ and $\tilde \psi(z)=\phi(z)A,$ for some $A \in M_n(\mathbb{C}).$
    \end{corollary}
    \begin{corollary}
        Assume that $\phi(z)$ and $\tilde \phi(z)$ are regular. Then, $H_\phi H_\psi=H^*_\psi H^*_\phi$ if and only if $\psi(z)=A\tilde \phi(z)$ and $\tilde \psi(z)=\phi(z)A,$ for some $A \in M_n(\mathbb{C}).$
    \end{corollary}
    \section{Solutions of operator equations}\label{section solutions of op eq}
    In the previous sections, we have seen results on skew commutators of Toeplitz and Hankel Operators on vector-valued Hardy spaces. In this section, we will see several types of skew commutators related to $S$ and $S^*$ defined on $H^2$.
        \begin{theorem}\label{th1}
            Let $S$ be the unilateral shift on $H^2.$ Then for any bounded linear operator $X$ on $H^2,$
            \begin{enumerate}
                \item  $SX=XS^*$ if and only if $X=0.$\label{p-1}
                \item $SX=X^*S^*$ if and only if $X=AS^*,$ for some self adjoint operator $A \in \mathcal{B}(H^2).$\label{p-2}
            \end{enumerate}
        \end{theorem}
    
        \begin{proof}
           Let $X$ satisfies $SX=XS^*$. Then $X$ also satisfies $X=S^*XS^*.$ Using the similar argument in Proposition \ref{prop1}, we get that $X=0.$
            Let $A$ be a self adjoint operator, and let $X=AS^*$,
            \begin{align*}
                SX=SAS^*=SA^*S^*=(AS^*)^*S^*=X^*S^*,
            \end{align*}
           This gives $X=AS^*,$ with $A=A^*$ is solution to the operator equation. Conversely, let $X$ be such that
            \begin{align*}
                SX=X^*S^* 
            \end{align*}
            Composing with $S^*$ on both sides, we get
            \begin{align}
                X=S^*X^*S^*
            \end{align}
            Let $A=XS$. Then
            \begin{equation*}
                A^*=S^*X^*=S^*(SXS)=XS=A.
            \end{equation*}
            and
            \begin{equation*}
                X=S^*X^*S^* =A^*S^*=AS^*.
            \end{equation*}
           \end{proof}
        Before proceeding to other operator equations involving the skew commutator of a unilateral shift, we look at the following:\\
         For any $f \in H^2,$ with $f(z)=\sum_{n=0}^{\infty} a_nz^n$, can be uniquely written as
        \begin{equation}
            f(z)=f_2(z^2)+zf_1(z^2),
        \end{equation}
        where $f_2(z)=\sum_{n=0}^{\infty} a_{2n}z^{n}$ and $f_1(z)=\sum_{n=0}^{\infty} a_{2n+1}z^{n}$. We define
        \begin{equation*}
            U:H^2\to H^2\oplus H^2 \text{ as } U(f)=(f_2,f_1).
        \end{equation*}
         It is easy to see that $U$ is unitary.
        \begin{theorem}\label{p-3}
           A bounded linear operator $X$ on $H^2$ satisfies 
           $$SX=X^*S$$
           if and only if 
           $$X=U^*\begin{pmatrix}
                 T^*_{\phi_{22}}&T_{\overline{f}+zf}\\
                 0&T_{\phi_{22}}
             \end{pmatrix}U,$$
           for some $ \phi_{22}\in H^{\infty}, f \in H^2 \text{ with } {\overline{f}+zf} \in L^\infty$ and $U$ is defined as above.
        \end{theorem}
        \begin{proof}
        The idea is to take the classical unilateral shift $S$ to an operator $\tilde S$ which unitarily equivalent to $S$ via unitary $U$ defined on $H^2 \oplus H^2,$ $i.e.$
        \begin{equation*}
            \tilde S= USU^*=
            \begin{pmatrix}
                0&S\\
                I&0
            \end{pmatrix}.
        \end{equation*}
        Let $Y=UXU^*.$ We solve the following operator equation.
        \begin{equation}
            \tilde S Y=Y^*\tilde S.
        \end{equation}
        Writing them in the block form with respect to the decomposition $H^2 \oplus H^2,$ we get
        \begin{equation*}
            \begin{pmatrix}
                 0&S\\
                I&0
            \end{pmatrix}
            \begin{pmatrix}
                Y_{11}&Y_{12}\\
                Y_{21}&Y_{22}
            \end{pmatrix}
            =
             \begin{pmatrix}
                Y^*_{11}&Y^*_{21}\\
                Y^*_{12}&Y^*_{22}
            \end{pmatrix}
            \begin{pmatrix}
                 0&S\\
                I&0
            \end{pmatrix}.
        \end{equation*}
        This gives rise to the following operator equations:
        \begin{align}
            SY_{21}=Y^*_{21}; \label{e.1} \\
            SY_{22}=Y^*_{11}S;\label{e.2} \\
            Y_{11}=Y^*_{22};\label{e.3} \\
            Y_{12}=Y^*_{12}S. \label{e.4}
        \end{align}
        To solve equation (\ref{e.1}) we take adjoint both the sides, we get
        \begin{equation*}
            Y_{21}=Y^*_{21}S^*.
        \end{equation*}
        This implies
        \begin{align*}
             Y_{21}=SY_{21}S^* .
        \end{align*}
        The fact that ${S^*}^n \xrightarrow{\text{SOT}} 0,$ we have $Y_{21}=0.$ From equation (\ref{e.2}) and equation (\ref{e.3}) we get $Y_{22}=T_{\phi_{22}},$ for some $\phi_{22} \in H^\infty.$
         To solve equation (\ref{e.4}), let us take its adjoint
         \begin{align*}
           Y_{12}=Y^*_{12}S.
         \end{align*}
         This implies 
         \begin{align*}
              Y_{12}= S^*Y_{12}S.
         \end{align*}
         Thus we have 
         \begin{align*}
             Y_{12}=T_{\phi_{12}}, \text{ for some } \phi_{12} \in L^\infty.
         \end{align*}
         Also,
         \begin{align*}
             T_{\phi_{12}}=T^*_{\phi_{12}}T_z=T_{\overline \phi_{12}z}.
         \end{align*}
         Therefore we have
         \begin{align*}
             \phi_{12}=z\overline{\phi_{12}}.
         \end{align*}
         The above condition gives that $\phi_{12}=\overline{f}+zf.$
         We get
         \begin{equation*}
         Y=
             \begin{pmatrix}
                 T^*_{\phi_{22}}&T_{\overline{f}+zf}\\
                 0&T_{\phi_{22}}
             \end{pmatrix},
         \end{equation*}
         where $\phi_{22}\in H^\infty,$ $f \in H^2$ with ${\overline{f}+zf}\in L^\infty.$ Conversely, note that
         \begin{align*}
             \begin{pmatrix}
                 0&S\\
                I&0
            \end{pmatrix}
            \begin{pmatrix}
                 T^*_{\phi_{22}}&T_{\overline{f}+zf}\\
                 0&T_{\phi_{22}}
             \end{pmatrix}
             &=\begin{pmatrix}
             0&ST_{\phi_{22}}\\
             T^*_{\phi_{22}}&T_{\overline{f}+zf}
             \end{pmatrix}\\
             &=\begin{pmatrix}
             0&T_{\phi_{22}}S\\
             T^*_{\phi_{22}}&T^*_{f+\overline {zf}}
             \end{pmatrix}\\
             &=\begin{pmatrix}
                 T_{\phi_{22}}&0\\
                 T^*_{\overline{f}+zf}&T^*_{\phi_{22}}
             \end{pmatrix}
             \begin{pmatrix}
                 0&S\\
                I&0
            \end{pmatrix}.
         \end{align*}
        \end{proof}
    The skew commutator of the unilateral shift is given by Nehari, more precisely,
    \begin{theorem}[Nehari] Let $X \in \mathcal{B}(H^2).$ Then
        $XS=S^*X$ if and only if $X=H_{\phi},$ for some $\phi \in L^\infty.$\label{p-4}
    \end{theorem}
    Motivated by this, let us find other skew commutators as well.
        \begin{theorem}\label{th2}
        Let $X \in \mathcal{B}(H^2)$ and let $P_{\langle 1 \rangle }$ be the orthogonal projection of $H^2$ onto $\langle 1 \rangle$. Then
            $$XS=S^*X^*$$ if and only if
                  \begin{equation}\label{p-6}
                      X= TP_{\langle 1 \rangle }+S^*A+BS^*,
                  \end{equation}
        for some $A,B$ and $T \in \mathcal{B}(H^2)$ with $A=A^*$ and $B=B^*.$

        \end{theorem}
        \begin{proof}
            Let $A,B$ and $T \in \mathcal{B}(H^2)$ with $A,B$ self adjoint and $X$ be as described in (\ref{p-6}). One can easily see that 
            \begin{align*}
                XS=S^*AS+B=S^*X^*.
            \end{align*}
            Conversely, let $XS=S^*X^*$ and let $A$ be any self-adjoint operator. Consider
            \begin{equation*}
                B:=(X-S^*A)S.
            \end{equation*}
            Clearly $B$ is self-adjoint. This gives
            \begin{align*}
                &XS=S^*AS+B
            \end{align*}
            Composing with $S^*$ we get
            \begin{align*}
                XSS^*=S^*ASS^*+BS^*.
            \end{align*}
            This gives, for $n \geq 1$
            \begin{align*}
                X(z^n)=(S^*A+BS^*)(z^n).
            \end{align*}
            So, we can write
            \begin{align*}
                X=TP_{\langle 1 \rangle } +S^*A+BS^*, \text{ where } T \in \mathcal{B}(H^2).
            \end{align*}
        \end{proof}
        
            The following characterizations can be derived by taking the adjoint and using the above theorems.
            \begin{remark}
            For a bounded linear operator $X \in \mathcal{B}(H^2),$ we have the following:
            \begin{enumerate}
            \item $XS=SX^*$ if and only if $X=U^*YU,$
            where
            \begin{equation*}
            Y=
           \begin{pmatrix}
                 T_{\phi_{22}}&0\\
                 T_{f+\overline{zf}}&T^*_{\phi_{22}}
             \end{pmatrix},
             \end{equation*}
             for some $ \phi_{22}\in H^{\infty},  f \in H^2 \text{ with } {\overline{f}+zf} \in L^\infty.$ 
             \item $S^*X=X^*S^*$ if and only if $X=U^*YU,$ where
                \begin{equation*}
                    Y=
                    \begin{pmatrix}
                 T^*_{\phi_{22}}&T_{\overline{f}+zf}\\
                 0&T_{\phi_{22}}
             \end{pmatrix},
                \end{equation*}
           for some $ \phi_{22}\in H^{\infty}, f \in H^2 \text{ with } {\overline{f}+zf} \in L^\infty.$
           \item $S^*X=X^*S$ if and only if
                    $X= P_{\langle 1 \rangle }T+AS+SB,$ for some $A,B$ and $T \in \mathcal{B}(H^2)$ with $A=A^*$ and $B=B^*.$
            \end{enumerate}
        \end{remark}
        
       Gu and Luo in their paper \cite{CGSL} have found the operator which intertwines $S_{\mathcal{H}_1}\oplus S_{\mathcal{H}_2}^*$ and $S_{\mathcal{H}_1}\oplus S_{\mathcal{H}_2}$ on the Vector valued Hardy space $H^2(\mathcal{H}_1)\oplus H^2(\mathcal{H}_2)$, which is precisely of the form
       \begin{equation*}
           \begin{pmatrix}
               T_{\phi_1}& T_{\phi_2}\\
               H_{\psi_1}& H_{\psi_2}
            \end{pmatrix}
       \end{equation*}
       for some 
       $\phi_{1} \in H^\infty(\mathcal{B}(\mathcal{H}_1)), \phi_{2} \in H^\infty(\mathcal{B}(\mathcal{H}_2,\mathcal{H}_1)),\psi_{1} \in  L^\infty (\mathcal{B}(\mathcal{H}_1,\mathcal{H}_2)) \text{ and } \psi_{2} \in L^\infty(\mathcal{B}(\mathcal{H}_2)).$ The natural question occurs: what is the skew commutator of $S_{\mathcal{H}_1}\oplus S_{\mathcal{H}_2}^*$?
       \begin{theorem}\label{skew of S+S*}
           Let $Y\in \mathcal{B}(H^2\oplus H^2).$ Then
           \begin{equation}\label{*}
               (S\oplus S^*)Y=Y(S\oplus S^*)^*
           \end{equation}
           if and only if 
           \begin{equation*}    
                Y=
                \begin{pmatrix}
                0&T_{\phi_{12}}\\
                T^*_{\phi_{21}}&H_\psi
                \end{pmatrix},
            \end{equation*}
             for some $\phi_{12}, \phi_{21} \in H^\infty \text{ and } \psi \in L^\infty.$
            where the block decomposition of $Y$ is with respect to $H^2\oplus H^2.$
            \end{theorem}
            \begin{proof}
               Let us write $Y$ in the form of an operator matrix in the decomposition $H^2\oplus H^2,$
           \begin{equation*}
           Y=
               \begin{pmatrix}
                   Y_{11}&Y_{12}\\
                   Y_{21}&Y_{22}
               \end{pmatrix}.
           \end{equation*}
           Now by equation (\ref{*}), we get
           \begin{equation*}
                \begin{pmatrix}
                   SY_{11}&SY_{12}\\
                   S^*Y_{21}&S^*Y_{22}
               \end{pmatrix}
               =
               \begin{pmatrix}
                   Y_{11}S^*&Y_{12}S\\
                   Y_{21}S^*&Y_{22}S
               \end{pmatrix}.
           \end{equation*}
            So, $Y_{11}=0$, there exist  $\phi_{12},~ \phi_{21} \in H^\infty$ and $\psi \in L^\infty $ such that $Y_{12}=T_{\phi_{12}},$ $Y_{21}=T^*_{\phi_{21}}$, and $Y_{22}=H_{\psi}.$ Hence
            \begin{equation*}    
                Y=
                \begin{pmatrix}
                0&T_{\phi_{12}}\\
                T^*_{\phi_{21}}&H_\psi
                \end{pmatrix}.
            \end{equation*}
        \end{proof}
        Using the similar argument of unitary $U,$ we have the following Lemma.
        \begin{lemma}\label{special}
            Let $X \in \mathcal{B}(H^2).$ Then 
            $$XS^2={S^*}^2X$$ if and only if $$X=U^*
                \begin{pmatrix}
                H_{\phi_{11}}&H_{\phi_{12}}\\
                H_{\phi_{21}}&H_{\phi_{22}}
                \end{pmatrix}U,$$ 
                for some $\phi_{11}, \phi_{12}, \phi_{21}, \phi_{22} \in L^{\infty}.$
        \end{lemma}
        \begin{theorem}\label{skew is S+S*}
           A bounded operator $Y\in \mathcal{B}(H^2\oplus H^2)$ 
           satisfies
           \begin{equation*}
               (S\oplus S^*)Y=Y^*(S\oplus S^*)
           \end{equation*}
           if and only if 
           \begin{equation}\label{converse skew}
               Y=
               \begin{pmatrix}
                   U^*T_{11}U& U^*\hat{T}_{12}U\\
                   U^*{\hat{T}_{21}}U& U^*T_{22}U
               \end{pmatrix}
           \end{equation}
           where 
           $$T_{11}=
             \begin{pmatrix}
                 T^*_{\phi_{11}}&T_{\overline{f}+zf}\\
                 0&T_{\phi_{11}}
             \end{pmatrix},
             T_{22}=
             \begin{pmatrix}
                T^*_{\phi_{22}}&T_{\overline{g}+zg}\\
                 0&T_{\phi_{22}}
             \end{pmatrix},
             $$
                for some $\phi_{11}, \phi_{22} \in H^\infty, f, g \in H^2$ with ${\overline{f}+zf}, {\overline{g}+zg} \in L^\infty,$ and
             $$
                \hat{T}_{12}=
                \begin{pmatrix}
                0&H_{\hat{\phi}_{12}}\\
                0&0
                \end{pmatrix},
                \hat{T}_{21}=
                \begin{pmatrix}
                0&H^*_{\hat{\phi}_{12}}\\
                0&0
                \end{pmatrix},
            $$
             for some ${\hat{\phi}_{12}} \in L^\infty.$
        \end{theorem}
        \begin{proof}
           Let us write $Y$ in the form of an operator matrix in the decomposition $H^2\oplus H^2$.
           \begin{equation*}
            Y=
               \begin{pmatrix}
                   Y_{11}&Y_{12}\\
                   Y_{21}&Y_{22}
               \end{pmatrix}.
           \end{equation*}
           The above operator equation leads to
           \begin{equation*}
                \begin{pmatrix}
                   SY_{11}&SY_{12}\\
                   S^*Y_{21}&S^*Y_{22}
               \end{pmatrix}
               =
               \begin{pmatrix}
                   Y^*_{11}S&Y^*_{21}S^*\\
                   Y^*_{12}S&Y^*_{22}S^*\\
               \end{pmatrix}.
           \end{equation*}
             By Theorem \ref{p-3}, 
             \begin{align*}
                 Y_{11}=U^*T_{11}U, \text{ where } T_{11}=
             \begin{pmatrix}
                 T^*_{\phi_{11}}&T_{\overline{f}+zf}\\
                 0&T_{\phi_{11}}
             \end{pmatrix},
             \end{align*}
             for some $\phi_{11} \in H^\infty, f \in H^2$ with $\overline{f}+zf \in L^\infty.$
             \begin{align*}
                 Y_{22}=U^*T_{22}U, \text{ where } T_{22}=
             \begin{pmatrix}
                 T^*_{\phi_{22}}&T_{\overline{g}+zg}\\
                 0&T_{\phi_{22}}
             \end{pmatrix},
             \end{align*}
             for some $\phi_{22} \in H^\infty, g \in H^2$ with $\overline{g}+zg\in L^\infty.$
             
             It remains to find the solutions to the following operator equations:
             \begin{equation}\label{equation-1}
                 SY_{12}=Y^*_{21}S^*,
             \end{equation}
             and 
             \begin{equation}\label{equation-2}
                 S^*Y_{21}=Y^*_{12}S.
             \end{equation}
            Composing equation (\ref{equation-1}) with $S$ we get
            \begin{equation*}
                SY_{12}S=Y^*_{21}.
            \end{equation*}
            Equivalently
            \begin{equation}\label{equation-3}
                Y_{21}=S^*Y^*_{12}S^*.
            \end{equation}
            Substituting in equation (\ref{equation-3}) in equation \ref{equation-2} we get
            \begin{equation} \label{equation-4}
                {S^*}^2Y^*_{12}S^*=Y^*_{12}S.
            \end{equation}
            Composing equation (\ref{equation-4}) with $S$, we get
            \begin{equation*}
                 Y^*_{12}S^2={S^*}^2Y^*_{12}.
            \end{equation*}
            Equivalently
            \begin{equation*}
                 Y_{12}S^2={S^*}^2Y_{12}.
            \end{equation*}
            By Lemma \ref{special}, we get
            $Y_{12}=U^*\hat{T}_{12}U,$ where
            \begin{equation*}
            \hat{T}_{12}=
                \begin{pmatrix}
                H_{\hat{\phi}_{11}}&H_{\hat{\phi}_{12}}\\
                H_{\hat{\phi}_{21}}&H_{\hat{\phi}_{22}}
                \end{pmatrix}, 
            \end{equation*}
             for some $\hat{\phi}_{11}, \hat{\phi}_{12}, \hat{\phi}_{21}, \hat{\phi}_{22} \in L^{\infty}.$
            Let $\hat T_{21}=UY_{21}U^*.$ Using the same analogy, we get
            \begin{align*}
                \hat T_{21}=UY_{21}U^*&= U{S}^*Y^*_{12}S^*U^*\\
                &=\tilde{S}^*{\hat{T}_{12}}^*\tilde{S}^*\\
                &=\begin{pmatrix}
                0&I\\
                S^*&0
            \end{pmatrix}
            \begin{pmatrix}
                H^*_{\hat{\phi}_{11}}&H^*_{\hat{\phi}_{21}}\\
                H^*_{\hat{\phi}_{12}}&H^*_{\hat{\phi}_{22}}
                \end{pmatrix}
            \begin{pmatrix}
                0&I\\
                S^*&0
            \end{pmatrix}\\
            &=\begin{pmatrix}
                H^*_{\hat{\phi}_{22}}S^*&H^*_{\hat{\phi}_{12}}\\
                S^*H^*_{\hat{\phi}_{21}}S^*&S^*H^*_{\hat{\phi}_{11}}
            \end{pmatrix}.
            \end{align*}
            We note that the equations (\ref{equation-1}) and (\ref{equation-2}) is equivalent to the following equations:
            \begin{align}\label{converted equation 1}
                \tilde S \hat T_{12}=\hat T^*_{21}{\tilde S}^*,
            \end{align}
            and
            \begin{align}\label{converted equation 2}
                 {\tilde S}^* \hat T_{21}=\hat T^*_{12}\tilde S.
            \end{align}
            Note that equation (\ref{converted equation 1}) implies
            \begin{align*}
                \begin{pmatrix}
                0&S\\
                I&0
            \end{pmatrix}
            \begin{pmatrix}
                H_{\hat{\phi}_{11}}&H_{\hat{\phi}_{12}}\\
                H_{\hat{\phi}_{21}}&H_{\hat{\phi}_{22}}
                \end{pmatrix}
                =\begin{pmatrix}
                H^*_{\hat{\phi}_{22}}S^*&H^*_{\hat{\phi}_{12}}\\
                S^*H^*_{\hat{\phi}_{21}}S^*&S^*H_{\hat{\phi}_{11}}
                \end{pmatrix}^*
                \begin{pmatrix}
                    0&I\\
                    S^*&0
                \end{pmatrix}.
            \end{align*}
            This implies
            \begin{align*}
               \begin{pmatrix}
                SH_{\hat{\phi}_{21}}&SH_{\hat{\phi}_{22}}\\
                H_{\hat{\phi}_{11}}&H_{\hat{\phi}_{12}}
                \end{pmatrix}=
                \begin{pmatrix}
                SH_{\hat{\phi}_{21}}SS^*&SH_{\hat{\phi}_{22}}\\
                H_{\hat{\phi}_{11}}SS^*&H_{\hat{\phi}_{12}}
            \end{pmatrix}.
            \end{align*}
            Therefore we have
            \begin{align}\label{E1}
                SH_{\hat{\phi}_{21}}=SH_{\hat{\phi}_{21}}SS^*, ~H_{\hat{\phi}_{11}}=H_{\hat{\phi}_{11}}SS^*.
            \end{align}
            Hence $H_{\hat{\phi}_{11}}=H_{\hat{\phi}_{21}}=0.$
            \begin{align}\label{E2}
                SH_{\hat{\phi}_{22}}=SH^*_{\hat{\phi}_{22}},
            \end{align}
            By equation (\ref{converted equation 2}) we have
            \begin{align*}
                \begin{pmatrix}
                    0&I\\
                    S^*&0
                \end{pmatrix}
                \begin{pmatrix}
                H^*_{\hat{\phi}_{22}}S^*&H^*_{\hat{\phi}_{12}}\\
                0&0
                \end{pmatrix}
                =\begin{pmatrix}
                0&0\\
                 H^*_{\hat{\phi}_{12}} &H^*_{\hat{\phi}_{22}}
                \end{pmatrix}
                \begin{pmatrix}
                    0&S\\
                    I&0
                \end{pmatrix}.
            \end{align*}
            Thus
            \begin{align*}
                \begin{pmatrix}
                    0&0\\
                    S^*H^*_{\hat{\phi}_{22}}S^*&S^*H^*_{\hat{\phi}_{11}}
                \end{pmatrix}=
                \begin{pmatrix}
                    0&0\\
                    H^*_{\hat{\phi}_{22}}&H^*_{\hat{\phi}_{12}}S
                \end{pmatrix}.
            \end{align*}
            Therefore we get $H_{\hat{\phi}_{22}}=0.$
            Thus
            \begin{align*}
                \hat{T}_{12}=
                \begin{pmatrix}
                0&H_{\hat{\phi}_{12}}\\
                0&0
                \end{pmatrix},
                \hat{T}_{21}=
                \begin{pmatrix}
                0&H^*_{\hat{\phi}_{12}}\\
                0&0
                \end{pmatrix},
            \end{align*}
            for some ${\hat{\phi}_{12}} \in L^\infty.$ Conversely let $Y$ be of the form as defined in equation (\ref{converse skew}). The diagonal equations follows from Theorem \ref{p-3}, and the off diagonal parts are immediate by definition of Hankel operator.
       \end{proof}
        \begin{theorem}\label{self adjoint S+S*}
             Let $Y\in \mathcal{B}(H^2\oplus H^2).$ Then
           \begin{equation*}
               (S\oplus S^*)Y=Y^*(S\oplus S^*)^*
           \end{equation*}
           if and only
           \begin{align}\label{self adjoin s+s star equation}
               Y=
               \begin{pmatrix}
              A_{11}S^*& S^*D^*S\\
              Y_{21}& P_{\langle 1\rangle}T+A_{22}S+SB_{22}
          \end{pmatrix},
           \end{align}
          for some bounded operator $Y_{21},$ and self adjoint operators $A_{11},A_{22}$ and $B_{22}$ on $H^2,$ where the operator matrix of $Y_{21}$ in the decomposition $\langle1\rangle \oplus zH^2$ of $H^2$ is given by
            \begin{equation*}
            Y=
                \begin{pmatrix}
                    a&g\\
                    0&D
                \end{pmatrix},
            \end{equation*}
            for some 
                $a \in \mathbb{C}, g\in \mathcal{B}(zH^2,\langle1\rangle)$ and $~ D\in \mathcal{B}(zH^2).$
            \end{theorem}
        \begin{proof}
            Let us write $Y$ in the form of an operator matrix in the decomposition $H^2\oplus H^2$.
           \begin{equation*}
               \begin{pmatrix}
                   Y_{11}&Y_{12}\\
                   Y_{21}&Y_{22}
               \end{pmatrix}.
           \end{equation*}
           Now by equation (\ref{*}), we get
           \begin{equation*}
                \begin{pmatrix}
                   SY_{11}&SY_{12}\\
                   S^*Y_{21}&S^*Y_{22}
               \end{pmatrix}
               =
               \begin{pmatrix}
                   Y^*_{11}S^*&Y^*_{21}S\\
                   Y^*_{12}S^*&Y^*_{22}S
               \end{pmatrix}.
           \end{equation*}
           From statement (\ref{p-2}) of Theorem \ref{th1}, we get
           \begin{equation*}
                Y_{11}=A_{11}S^*, \text{ for some self adjoint } A_{11}\in \mathcal{B}(H^2).
           \end{equation*}
           Also, 
           \begin{equation*}
               Y_{22}=P_{\langle 1\rangle}T+A_{22}S+SB_{22},
           \end{equation*}
           for some $T,A_{22},B_{22} \in \mathcal{B}(H^2),$ with $A^*_{22}=A_{22}$ and $B^*_{22}=B_{22}.$
           We have,
           \begin{align*}
              S^*Y_{21} =Y^*_{12}S^*
              =S^*Y_{21}SS^*.
            \end{align*}
            This gives $S^*Y_{21}P_{\langle 1\rangle}=0.$
            Equivalently the operator matrix of $Y_{21}$ in the decomposition $\langle1\rangle \oplus zH^2$ of $H^2$ is given by
            \begin{equation*}
                \begin{pmatrix}
                    a&g\\
                    0&D
                \end{pmatrix},\text{ for some } g\in \mathcal{B}(zH^2,\langle1\rangle),~ D\in \mathcal{B}(zH^2).
            \end{equation*}
            Also, it is easy to see that $Y_{12}=S^*D^*S.$ Conversely let $Y$ as defined in equation (\ref{self adjoin s+s star equation}). We note that
            \begin{align*}
            (S \oplus S^*) Y&=
                \begin{pmatrix}
                    S&0\\
                    0&S^*
                \end{pmatrix}
                \begin{pmatrix}
                    A_{11}S^*& S^*D^*S\\
                    Y_{21}& P_{\langle 1\rangle}T+A_{22}S+SB_{22}
                \end{pmatrix}\\
                &=
                \begin{pmatrix}
                    SA_{11}S^*& SS^*D^*S\\
                    S^*Y_{21}& S^*P_{\langle 1\rangle}T+S^*A_{22}S+B_{22}
                \end{pmatrix}\\
                &= \begin{pmatrix}
                    SA_{11}S^*& SY_{12}\\
                    S^*Y_{21}& S^*P_{\langle 1\rangle}T+S^*A_{22}S+B_{22}
                \end{pmatrix}.
            \end{align*}
            By definition of $Y_{21}$, we have 
            \begin{align*}
                S^*Y_{21}=S^*Y_{21}SS^*=S^*DSS^*,
            \end{align*}
            and for $n \geq 0,$
            \begin{align*}
            Y^*_{21}S(z^n)=D^*S(z^n)=SS^*D^*S(z^n).
            \end{align*}
            Therefore we get
            \begin{align*}
                (S \oplus S^*) Y&=
                \begin{pmatrix}
                    SA_{11}S^*& Y^*_{21}S\\
                    SDSS^*& S^*P_{\langle 1\rangle}T+S^*A_{22}S+B_{22}
                \end{pmatrix}\\
                &=Y^*(S\oplus S^*)^*.
            \end{align*}
\end{proof}

\section*{{{Acknowledgment}}}
P. Aroda is supported by NBHM Ph.D. fellowship, File No: 0203/13(18)/2021-R\&D-II/13138, by the Department of Atomic Energy, Government of India.

   \noindent $^{*}$ [ P. Aroda ] Department of Mathematics, Indian Institute of Technology Bombay, Mumbai, 400076, India.\\
		\textit{Email address:}~ priyanka.aroda@iitb.ac.in, rspriyanka0412@gmail.com.
        
	\noindent $^{**}$[ S. Dey ] Department of Mathematics, Indian Institute of Technology Bombay, Mumbai, 400076, India.\\
        \textit{Email address:}~ santanudey@iitb.ac.in, santanudey76@gmail.com
\end{document}